\newfont{\bcb}{msbm10}
\newfont{\matb}{cmbx10}
\newfont{\got}{eufm10}

\documentclass[12pt]{amsart}
\usepackage{amsmath, amsthm, amscd, amsfonts, amssymb, latexsym, graphicx, color}
\usepackage[bookmarksnumbered, colorlinks, plainpages, hypertex]{hyperref}

\usepackage[cp1250]{inputenc}

\usepackage{amsmath,amsthm}
\usepackage{amssymb,latexsym}
\usepackage{enumerate}

\newtheorem{theorem}{Theorem}[section]
\newtheorem{lemma}[theorem]{Lemma}
\newtheorem{proposition}[theorem]{Proposition}
\newtheorem{corollary}[theorem]{Corollary}
\theoremstyle{definition}

\theoremstyle{remark}
\newtheorem{remark}[theorem]{Remark}
\numberwithin{equation}{section}

\begin{document}

\title[An analytic closedness theorem]{A closedness theorem
      \\ over Henselian valued fields
      \\ with analytic structure}

\author[Krzysztof Jan Nowak]{Krzysztof Jan Nowak}


\subjclass[2000]{Primary 32B05, 32B20, 14G27; Secondary 03C10,
14P15, 12J25.}

\keywords{Henselian valued fields, analytic structure, separated
power series, closedness theorem, fiber shrinking, b-minimal cell
decomposition}

\date{}

\begin{abstract}
The main purpose of the paper is to establish a closedness theorem
over Henselian valued fields $K$ of equicharacteristic zero (not
necessarily algebraically closed) with separated analytic
structure. It says that every projection with a projective fiber
is a definably closed map. This remains valid also for valued
fields with analytic structure induced by a strictly convergent
Weierstrass systems, including the classical, complete rank one
valued fields with the Tate algebra of strictly convergent power
series. As application, we prove two theorems on existence of the
limit and on piecewise continuity.
\end{abstract}

\maketitle

\section{Introduction}

Throughout the paper, we shall deal with Henselian valued fields
$K$ with separated analytic structure, nor necessarily
algebraically closed. We shall always assume that the ground field
$K$ is of equicharacteristic zero. A separated analytic structure
is determined by a certain separated  Weierstrass system
$\mathcal{A}$ defined on an arbitrary commutative ring $A$ with
unit (cf.~\cite{C-Lip-0,C-Lip}), and the involved analytic
language $\mathcal{L}$ is the two sorted, semialgebraic language
$\mathcal{L}_{Hen}$ augmented by the reciprocal function $1/x$ and
the names of all functions of the system $\mathcal{A}$, construed
via the analytic $\mathcal{A}$-structure on their natural domains
and as zero outside them. For convenience, we remind the reader of
these concepts in Section~2. The theory of valued fields with
analytic structure was developed in the
papers~\cite{Dries,Dries-Has,L-R,C-Lip-R,C-Lip-0,C-Lip}.

\vspace{1ex}

Given a valued field $K$, denote by $v$, $\Gamma = \Gamma_{K}$,
$K^{\circ}$, $K^{\circ \circ}$ and $\widetilde{K}$ the valuation,
its value group, the valuation ring, maximal ideal and residue
field, respectively. By the $K$-topology on $K^{n}$ we mean the
topology induced by the valuation $v$.

\vspace{1ex}

The main result of this article is the following closedness
theorem.

\begin{theorem}\label{clo-th}
Given an $\mathcal{L}$-definable subset $D$ of $K^{n}$, the
canonical projection
$$ \pi: D \times (K^{\circ})^{m} \longrightarrow D  $$
is definably closed in the $K$-topology, i.e.\ if $B \subset D
\times (K^{\circ})^{m}$ is an $\mathcal{L}$-definable closed
subset, so is its image $\pi(B) \subset D$.
\end{theorem}

It immediately yields the five corollaries stated below. The last
three of them enable, in the non-Archimedean analytic case,
application of resolution of singularities and transformation to a
normal crossing by blowing up in much the same way as over locally
compact ground fields (see \cite{Now-Sel,Now-thm} for application
in the non-Archimedean algebraic case).

\begin{corollary}\label{clo-th-cor-0}
Let $D$ be an $\mathcal{L}$-definable subset of $K^{n}$ and
$\,\mathbb{P}^{m}(K)$ stand for the projective space of dimension
$m$ over $K$. Then the canonical projection
$$ \pi: D \times \mathbb{P}^{m}(K) \longrightarrow D $$
is definably closed. \hspace*{\fill} $\Box$
\end{corollary}

\begin{corollary}\label{clo-th-cor-1}
Let $A$ be a closed $\mathcal{L}$-definable subset of
$\,\mathbb{P}^{m}(K)$ or $R^{m}$. Then every continuous
$\mathcal{L}$-definable map $f: A \to K^{n}$ is definably closed
in the $K$-topology.
\end{corollary}


\begin{corollary}\label{clo-th-cor-2}
Let $\phi_{i}$, $i=0,\ldots,m$, be regular functions on $K^{n}$,
$D$ be an $\mathcal{L}$-definable subset of $K^{n}$ and $\sigma: Y
\longrightarrow K\mathbb{A}^{n}$ the blow-up of the affine space
$K\mathbb{A}^{n}$ with respect to the ideal
$(\phi_{0},\ldots,\phi_{m})$. Then the restriction
$$ \sigma: Y(K) \cap \sigma^{-1}(D) \longrightarrow D $$
is a definably closed quotient map.
\end{corollary}


\begin{proof} Indeed, $Y(K)$ can be regarded as a closed algebraic subvariety of
$K^{n} \times \mathbb{P}^{m}(K)$ and $\sigma$ as the canonical
projection.
\end{proof}

\begin{corollary}\label{clo-th-cor-3}
Let $X$ be a smooth $K$-variety, $\phi_{i}$, $i=0,\ldots,m$,
regular functions on $X$, $D$ be an $\mathcal{L}$-definable subset
of $X(K)$ and $\sigma: Y \longrightarrow X$ the blow-up of the
ideal $(\phi_{0},\ldots,\phi_{m})$. Then the restriction
$$ \sigma: Y(K) \cap \sigma^{-1}(D) \longrightarrow D $$
is a definably closed quotient map.  \hspace*{\fill} $\Box$
\end{corollary}


\begin{corollary}\label{clo-th-cor-4} (Descent property)
Under the assumptions of the above corollary, every continuous
$\mathcal{L}$-definable function
$$ g: Y(K) \cap \sigma^{-1}(D) \longrightarrow K $$
that is constant on the fibers of the blow-up $\sigma$ descends to
a (unique) continuous $\mathcal{L}$-definable function $f: D
\longrightarrow K$. \hspace*{\fill} $\Box$
\end{corollary}

The closedness theorem will be proven in Section~3. The strategy
of proof in the analytic settings will generally follow the one in
the algebraic case from my papers~\cite{Now-Sel,Now-thm}. We rely,
in particular, on fiber shrinking and the local behavior of
definable functions of one variable. Again, we make use of
relative quantifier elimination for ordered abelian groups (in a
many-sorted language with imaginary auxiliary sorts) due to
Cluckers--Halupczok~\cite{C-H}. But now we apply elimination of
valued field quantifiers for the theory $T_{Hen,\mathcal{A}}$ and
b-minimal cell decompositions with centers (cf.~\cite{C-L}).


\begin{remark}
The closedness theorem holds also for analytic structures induced
by strictly convergent Weierstrass systems, because every such
structure can be extended in a definitional way (extension by
Henselian functions) to a separated analytic structure
(cf.~\cite{C-Lip}). Examples of such structures are the classical,
complete rank one valued fields with the Tate algebra of strictly
convergent power series.
\end{remark}

In Section~4, we give some applications of the closedness theorem,
namely two theorems on existence of the limit
(Proposition~\ref{limit-th2}) and on piecewise continuity
(Theorem~\ref{piece}). Note finally that our proof of the
closedness theorem makes use of a certain version of the former
result (Proposition~\ref{limit-th1}).



\section{Fields with analytic structure}
In this section we recall the concept of an analytic structure
(cf.~\cite[Section~4.1]{C-L}). Let $A$ be a commutative ring with
unit and with a fixed proper ideal $I \varsubsetneq A$. A {\em
separated} $(A, I)$-{\em system} is a certain system $\mathcal{A}$
of $A$-subalgebras $A_{m,n} \subset A[[\xi,\rho]]$, $m,n \in
\mathbb{N}$; here $A_{0,0} = A$. Two kinds of variables, $\xi$ and
$\rho$, play different roles. Roughly speaking, the variables
$\xi$ vary over the valuation ring (or the closed unit disc)
$K^{\circ}$ of a valued field $K$, and the variables $\rho$ vary
over the maximal ideal (or the open unit disc) $K^{\circ \circ}$
of $K$. $\mathcal{A}$ is called a {\em separated pre-Weierstrass
system} if two usual Weierstrass division theorems hold in each
$A_{m,n}$. When, in addition, such a pre-Weierstrass system
$\mathcal{A}$ satisfies a condition referring to the so-called
rings of $\mathcal{A}$-fractions, it is called a {\em separated
Weierstrass system} (loc.~cit.). This condition may be regarded as
a kind of weak Noetherian property, because it implies, in
particular, that if
$$ f = \sum_{\mu,\nu} \, a_{\mu \nu} \, \xi^{\mu} \rho^{\nu} \in
   A_{m,n}, $$
then the ideal of $A$ generated by the $a_{\mu \nu}$ is finitely
generated.

\vspace{1ex}

Let $\mathcal{A}$ be a separated Weierstrass system and $K$ be a
valued field. A {\em separated analytic} $\mathcal{A}$-{\em
structure} on the field K (loc.~cit.) is a collection of
homomorphisms $\sigma_{m,n}$ from $A_{m,n}$ to the ring of
$K^{\circ}$-valued functions on $(K^{\circ})^{m} \times (K^{\circ
\circ})^{n}$, $m,n \in \mathbb{N}$, such that

1) $\sigma_{0,0} (I) \subset K^{\circ \circ}$;

2) $\sigma_{m,n}(\xi_{i})$ and $\sigma_{m,n}(\rho_{j})$ are the
$i$-th and $(m+j)$-th coordinate functions on $(K^{\circ})^{m}
\times (K^{\circ \circ})^{n}$, respectively;

3) $\sigma_{m+1,n}$ and $\sigma_{m,n+1}$ extend $\sigma_{m,n}$,
where functions on $(K^{\circ})^{m} \times (K^{\circ \circ})^{n}$
are identified with those functions on
$$ (K^{\circ})^{m+1} \times (K^{\circ \circ})^{n} \ \ \ \text{or} \
   \ \ (K^{\circ})^{m} \times (K^{\circ \circ})^{n+1} $$
which do not depend on the coordinate $\xi_{m+1}$ or $\rho_{n+1}$,
respectively.

\vspace{1ex}

Further, consider a separated pre-Weierstrass $(A, I)$-system
$\mathcal{A}$ and assume that $A = F^{\circ}$ and $I = F^{\circ
\circ}$ for a valued field $F$. Then $\mathcal{A}$ is a
Weierstrass system iff for every $f \in A_{m,n}$, $f \neq 0$, $m,n
\in \mathbb{N}$, there is an element $c \in F$ such that $cf \in
A_{m,n}$ and the Gauss norm (which is then well defined) $\| cf \|
= 1$ (loc.~cit.).

\vspace{1ex}

Now let us recall some properties of analytic structures. Analytic
$\mathcal{A}$-structures preserve composition (op.~cit.,
Proposition~4.5.3). If the ground field $K$ is non-trivially
valued, then the function induced by a power series from
$A_{m,n}$, $m,n \in \mathbb{N}$, is the zero function iff the
image in $K$ of each of its coefficients is zero (op.~cit.,
Proposition~4.5.4).

\vspace{1ex}

\begin{remark}\label{ext-par}
When considering a particular field $K$ with analytic
$\mathcal{A}$-structure, one may assume that $\mathrm{ker}\,
\sigma_{0,0} = (0)$. Indeed, replacing $A$ by $A/\mathrm{ker}\,
\sigma_{0,0}$ yields an equivalent analytic structure on $K$ with
this property. Then $A = A_{0,0}$ can be regarded as a subring of
$K^{\circ}$. Moreover, by extension of parameters, one can get a
(unique) separated Weierstrass system $\mathcal{A}(K)$ over
$(K^{\circ},K^{\circ \circ})$ and $K$ has separated analytic
$\mathcal{A}(K)$-structure; a similar extension of parameters can
be performed for any subfield $F \subset K$ of parameters
(op.~cit., Theorem~4.5.7  f.f.). Further, a separated analytic
$\mathcal{A}$-structure on a valued field $K$ can be uniquely
extended to any algebraic exteksion $K'$ of $K$; in particular, to
the algebraic closure $K_{alg}$ of $K$ (op.~cit., Theorem~4.5.11).
The forgoing properties remain valid in the case of strictly
convergent Weierstrass systems too. Finally, every valued field
with separated analytic structure is Henselian (op.~cit.,
Proposition~4.5.10).
\end{remark}

\vspace{1ex}

Now we can describe the analytic language $\mathcal{L}$ of an
analytic structure $K$ determined by a separated Weierstrass
system $\mathcal{A}$. We begin by defining the semialgebraic
language $\mathcal{L}_{Hen}$. It is a two sorted language with the
main, valued field sort $K$, and the auxiliary $RV$-sort
$$ RV = RV(K) := RV^{*} \cup \{ 0 \}, \ \ \
   RV^{*}(K) := K^{\times}/(1 + K^{\circ \circ}); $$
here $A^{\times}$ denotes the set of units of a ring $A$. The
language of the valued field sort is the language of rings
$(0,1,+,-, \cdot)$. The language of the auxiliary sort is the
so-called inclusion language (op.~cit., Section~6.1). The only map
connecting the sorts is the canonical map
$$ rv: K \to RV(K), \ \ \ 0 \mapsto 0. $$
Since
$$ \widetilde{K}^{\times} \simeq (K^{\circ})^{\times}/(1 + K^{\circ \circ})
   \ \ \ \text{and} \ \ \ \Gamma \simeq K^{\times}/(K^{\circ})^{\times}, $$
we get the canonical exact sequence
$$ 1 \to \widetilde{K} \to RV(K) \to \Gamma \to 0. $$
This sequence splits iff the valued field $K$ has an angular
component map.

\vspace{1ex}

The analytic language $\mathcal{L} =
\mathcal{L}_{Hen,\mathcal{A}}$ is the semialgebraic language
$\mathcal{L}_{Hen}$ augmented on the valued field sort $K$ by the
reciprocal function $1/x$ (with $1/0 :=0$) and the names of all
functions of the system $\mathcal{A}$, together with the induced
language on the auxiliary sort $RV$ (op.~cit., Section~6.2). A
power series $f \in A_{m,n}$ is construed via the analytic
$\mathcal{A}$-structure on their natural domains and as zero
outside them. More precisely, $f$ is interpreted as a function
$$ \sigma (f): (K^{\circ})^{m} \times (K^{\circ \circ})^{n} \to
   K^{\circ}, $$
extended by zero on $K^{m+n} \setminus (K^{\circ})^{m} \times
(K^{\circ \circ})^{n}$.

\vspace{1ex}

In the equicharacteristic case, however, the induced language on
the auxiliary sort $RV$ further coincides with the semialgebraic
inclusion language. It is so because then \cite[Lemma~6.3.12]{C-L}
can be strengthen as follows, whereby \cite[Lemma~6.3.14]{C-L} can
be directly reduced to its algebraic analogue. Consider a strong
unit on the open ball $B = K_{alg}^{\circ \circ}$. Then $rv
(E^{\sigma})(x)$ is constant when $x$ varies over $B$. This is no
longer true in the mixed characteristic case. There, a weaker
conclusion asserts that the functions $rv_{n} (E^{\sigma})(x)$, $n
\in \mathbb{N}$, depend only on $rv_{n}(x)$ when $x$ varies over
$B$; actually, $rv_{n} (E^{\sigma})(x)$ depend only on $\: x \! \!
\mod (n \cdot K^{\circ \circ}_{alg})$ when $x$ varies over $B$, as
indicated in~\cite[Remark~A.1.12]{C-Lip}. Under the circumstances,
the residue field $\widetilde{K}$ is orthogonal to the value group
$\Gamma_{K}$, whenever the ground field $K$ has an angular
component map or, equivalently, when the auxiliary sort $RV$
splits (in a non-canonical way):
$$ RV(K) \simeq \widetilde{K} \times \Gamma_{K}. $$
This means that every definable set in the auxiliary sort $RV(K)$
is a finite union of the Cartesian products of some sets definable
in the residue field sort $\widetilde{K}$ (in the language of
rings) and in the value group sort $\Gamma_{K}$ (in the language
of ordered groups). The orthogonality property will often be used
in the paper, similarly as it was in the algebraic case treated in
our papers~\cite{Now-Sel,Now-thm}.

\begin{remark}
Not all valued fields $K$ have an angular component map, but it
exists if $K$ has a cross section, which happens whenever $K$ is
$\aleph_{1}$-saturated (cf.~\cite[Chap.~II]{Ch}). Moreover, a
valued field $K$ has an angular component map whenever its residue
field $\Bbbk$ is $\aleph_{1}$-saturated
(cf.~\cite[Corollary~1.6]{Pa2}). In general, unlike for $p$-adic
fields and their finite extensions, adding an angular component
map does strengthen the family of definable sets. Since the
$K$-topology is $\mathcal{L}$-definable, the closedness theorem is
a first order property. Therefore it can be proven using
elementary extensions, and thus one may assume that an angular
component map exists.
\end{remark}

Let $\mathcal{T}_{Hen,\mathcal{A}}$ be the theory of all Henselian
valued fields of characteristic zero with analytic
$\mathcal{A}$-structure. The crucial result about analytic
structures is the following \cite[Theorem~6.3.7]{C-L}.

\begin{theorem}\label{ball}
The theory $\mathcal{T}_{Hen,\mathcal{A}}$ eliminates valued field
quantifiers, is b-minimal with centers and preserves all balls.
Moreover, $\mathcal{T}_{Hen,\mathcal{A}}$ has the Jacobian
property. \hspace*{\fill} $\Box$
\end{theorem}

Therefore the theory $\mathcal{T}_{Hen,\mathcal{A}}$ admits
b-minimal cell decompositions with centers (cf.~\cite{C-L}).

\section{Proof of the closedness theorem}

From now on we shall assume that the ground field $K$ with
separated analytic structure $\mathcal{A}$ is of
equicharacteristic zero, and that $K$ has an angular component
map. In the algebraic case, the proofs of the closedness theorem
given in our papers~\cite{Now-Sel,Now-thm}) make use of the
following three main tools: the theorem on existence of the limit
(\cite[Proposition~5.2]{Now-Sel} and \cite[Theorem~5.1]{Now-thm}),
fiber shrinking (\cite[Proposition~6.1]{Now-Sel,Now-thm}) and cell
decomposition in the sense of Pas.

\vspace{1ex}

Fiber shrinking was reduced, by means of elimination of valued
field quantifiers, to Lemma~\ref{line-1} below
(\cite[Lemma~6.2]{Now-thm}), which, in turn, was obtained via
relative quantifier elimination for ordered abelian groups. That
approach can be repeated verbatim in the analytic settings.

\begin{lemma}\label{line-1}
Let $\Gamma$ be an ordered abelian group and $P$ be a definable
subset of $\Gamma^{n}$. Suppose that $(\infty,\ldots,\infty)$ is
an accumulation point of $P$, i.e.\ for any $\delta \in \Gamma$
the set
$$ \{ x \in P: x_{1} > \delta, \ldots, x_{n} > \delta \} \neq \emptyset $$
is non-empty. Then there is an affine semi-line
$$ L = \{ (r_{1}t + \gamma_{1},\ldots,r_{n}t + \gamma_{n}): \, t
   \in \Gamma, \ t \geq 0 \} \ \ \ \text{with} \ \ r_{1},\ldots,r_{n}
   \in \mathbb{N}, $$
passing through a point $\gamma = (\gamma_{1},\ldots,\gamma_{n})
\in P$ and such that $(\infty,\ldots,\infty)$ is an accumulation
point of the intersection $P \cap L$ too. \hspace*{\fill} $\Box$
\end{lemma}

In a similar manner, one can obtain the following

\begin{lemma}\label{line-2}
Let $P$ be a definable subset of $\Gamma^{n}$ and
$$ \pi: \Gamma^{n} \to \Gamma, \ \ \ (x_{1},\ldots,x_{n}) \mapsto
   x_{1} $$
be the projection onto the first factor. Suppose that $\infty$ is
an accumulation point of $\pi(P)$. Then there is an affine
semi-line
$$ L = \{ (r_{1}t + \gamma_{1},\ldots,r_{n}t + \gamma_{n}): \, t
   \in \Gamma, \ t \geq 0 \} \ \ \text{with} \ \ r_{1},\ldots,r_{n}
   \in \mathbb{N}, \, r_{1} >0, $$
passing through a point $\gamma = (\gamma_{1},\ldots,\gamma_{n})
\in P$ and such that $\infty$ is an accumulation point of $\pi(P
\cap L)$ too.
\end{lemma}

In this paper, however, a suitable analytic version of the theorem
on existence of the limit and application of b-minimal cell
decompositions require some new ideas and work. Actually, we first
prove a certain version of the former
(Proposition~\ref{limit-th1}), making use of the theorem on term
structure (\cite[Theorem~6.3.8]{C-L}), recalled below. Another
analytic version (Proposition~\ref{limit-th2}) will be established
in Section~4 by means of the closedness theorem. In further
reasonings, we shall often make use of Lemmas~\ref{line-1}
and~\ref{line-2}.

\vspace{1ex}

Denote by $\mathcal{L}^{*}$ the analytic language $\mathcal{L}$
augmented by all Henselian functions
$$ h_{m}: K^{m+1} \times RV(K) \to K,  \ \ m \in \mathbb{N}, $$
which are defined by means of a version of Hensel's lemma
(cf.~\cite{C-L}, Section~6.1).

\begin{theorem}\label{term}
Let $K$ be a Henselian field with analytic
$\mathcal{A}$-structure. Let $f: X \to K$, $X \subset K^{n}$, be
an $\mathcal{L}(B)$-definable function for some set of parameters
$B$. Then there exist an $\mathcal{L}(B)$-definable function $g: X
\to S$ with $S$ auxiliary and an $\mathcal{L}^{*}(B)$-term $t$
such that
$$ f(x) = t(x,g(x)) \ \ \ \text{for all} \ \ x \in X. $$
\end{theorem}
\hspace*{\fill} $\Box$

We turn to the following analytic version of the theorem on
existence of the limit, which also may be regarded as a version of
Puiseux's theorem.

\begin{proposition}\label{limit-th1}
Let $f:E \to K$ be an $\mathcal{L}$-definable function on a subset
$E$ of $K$ and suppose $0$ is an accumulation point of $E$. Then
there is an $\mathcal{L}$-definable subsets $F \subset E$ with
accumulation point $0$ and a point $w \in \mathbb{P}^{1}(K)$ such
that
$$ \lim_{x \rightarrow 0}\, f|F\, (x) = w. $$
Moreover, we can require that
$$ \{ (x,f(x)): x \in F \} \subset \{ (x^{r}, \phi(x)): x \in G
   \}, $$
where $r$ is a positive integer and $\phi$ is a definable
function, a composite of some functions induced by series from
$\mathcal{A}$ and of some algebraic power series (coming from
Henselian functions $h_{m}$). Then, in particular, the definable
set
$$ \{ (v(x), v(f(x))): \; x \in (F \setminus \{0 \} \}
   \subset \Gamma \times (\Gamma \cup \{\infty \}) $$
is contained in an affine line with rational slope
$$ l = \frac{p}{q} \cdot k + \beta,  $$
with $p,q \in \mathbb{Z}$, $q>0$, $\beta \in \Gamma$, or in\/
$\Gamma \times \{ \infty \}$.
\end{proposition}

\begin{proof}
In view of Remark~\ref{ext-par}, we may assume that $K$ has
separated analytic $\mathcal{A}(K)$-structure. We apply
Theorem~\ref{term} and proceed with induction with respect to the
complexity of the term $t$. Since an angular component map exists,
the sorts $\widetilde{K}$ and $\Gamma$ are orthogonal in
$$ RV(K) \simeq \widetilde{K} \times \Gamma_{K}. $$
Therefore, after shrinking $F$, we can assume that
$\overline{ac}\: (F) = \{1 \}$ and the function $g$ goes into $ \{
\xi \} \times \Gamma^{s} $ with a $\xi \in \widetilde{K}^{s}$, and
next that $\xi = (1,\ldots,1)$; similar reductions were considered
in our papers~\cite{Now-Sel,Now-thm}. For simplicity, we look at
$g$ as a function into $\Gamma^{s}$. We shall briefly explain the
most difficult case where
$$ t(x,g(x)) = h_{m}(a_{0}(x),\ldots,a_{m}(x),g_{0}(x)), $$
assuming that the theorem holds for the terms
$a_{0},\ldots,a_{m}$; here $g_{0}$ is one of the components of
$g$. By Lemma~\ref{line-2}, we can assume that
\begin{equation}
p v(x) + q g_{0}(x) + v(a) = 0
\end{equation}
for some $p,q \in \mathbb{Z}$, $a \in K \setminus \{ 0 \}$. By the
induction hypothesis, we get
$$ \{ (x,a_{i}(x)): x \in F \} \subset \{ (x^{r}, \alpha_{i}(x)): x \in
   G \}, \ \ \ i = 0,1,\ldots,m. $$
Put
$$ P(x,T) := \sum_{i=0}^{m} \, a_{i}(x) T^{i}. $$
By the very definition of $h_{m}$ and since we are interested in
the vicinity of zero, we may assume that there is
$i_{0}=0,\ldots,m$ such that
$$ \forall \: x \in F \ \exists \: u \in K \ \
   v(u) = g_{0}(x), \ \ \overline{ac} \: u =1, $$
\begin{equation}
  v(a_{i_{0}}(x)u^{i_{0}}) = \min \, \{ v(a_{i}(x)u^{i}), \
i=1,\ldots,m \},
\end{equation}
$$ v(P(x,u)) > v(a_{i_{0}}(x)u^{i_{0}}), \ \ v \left( \frac{\partial
   \, P}{\partial \, T} (x,u) \right) = v(a_{i_{0}}(x)u^{i_{0}}). $$
Then $h_{m}(a_{0}(x),\ldots,a_{m}(x),g_{0}(x))$ is a unique $b(x)
\in K$ such that
$$ P(x,b(x))=0, \ \ v(b(x)) = g_{0}(x), \ \ \overline{ac} \: b(x)
   =1. $$
By \cite[Remarks~7.2, 7.3]{Now-thm}, the set $F$ contains the set
of points of the form $c^{r}t^{Nqr}$ for some $c \in K$ with
$\overline{ac} \: c=1$, a positive integer $N$ and all $t \in
K^{\circ}$ with $\overline{ac} \: t =1$. Hence and by
equation~(3.1), we get
$$ g_{0}(c^{r}t^{Nqr}) = g_{0}(c^{r}) - v(t^{Npr}). $$
Take $d \in K$ such that $g_{0}(c^{r}) =v(d)$ and $\overline{ac}
\: d=1$. Then
$$ g_{0}(c^{r}t^{Nqr}) = v(dt^{-Npr}). $$
Thus the homothetic change of variable
$$ Z = T/dt^{-Npr} = t^{Npr}T/d $$
transforms the polynomial
$$ P(c^{r}t^{Nqr},T) = \sum_{i=0}^{m} \, \alpha_{i}(ct^{Nq}) T^{i} $$
into a polynomial $Q(t,Z)$ to which Hensel's lemma applies
(cf.~\cite[Lemma~3.5]{Pa1}):

\begin{equation}
  P(c^{r}t^{Nqr},T) = P(c^{r}t^{Nqr},dt^{-Npr}Z) =
  \end{equation}
$$ \alpha_{i_{0}}(ct^{Nq}) \cdot (dt^{-Npr})^{i_{0}} \cdot Q(t,Z).
$$
Indeed, the formulas (3.2) imply that the coefficients of the
polynomial $Q$ are power series (of order $\geq 0$) in the
variable $t$, and that
$$ v(Q(0,1)) > 0 \ \ \ \text{and} \ \ \ v\left( \frac{\partial \,
   Q}{\partial \, Z} (0,1) \right) = 0. $$
Therefore the conclusion of the theorem follows.
\end{proof}

We still need the concept of fiber shrinking introduced in our
paper~\cite{Now-Sel}. Let $A$ be an $\mathcal{L}$-definable subset
of $K^{n}$ with accumulation point
$$ a = (a_{1},\ldots,a_{n}) \in K^{n} $$
and $E$ an $\mathcal{L}$-definable subset of $K$ with accumulation
point $a_{1}$. We call an $\mathcal{L}$-definable family of sets
$$ \Phi = \bigcup_{t \in E} \ \{ t \} \times \Phi_{t} \subset A $$
an $\mathcal{L}$-definable $x_{1}$-fiber shrinking for the set $A$
at $a$ if
$$ \lim_{t \rightarrow a_{1}} \, \Phi_{t} = (a_{2},\ldots,a_{n}),
$$
i.e.\ for any neighbourhood $U$ of $(a_{2},\ldots,a_{n}) \in
K^{n-1}$, there is a neighbourhood $V$ of $a_{1} \in K$ such that
$\emptyset \neq \Phi_{t} \subset U$ for every $t \in V \cap E$, $t
\neq a_{1}$. When $n=1$, $A$ is itself a fiber shrinking for the
subset $A$ of $K$ at an accumulation point $a \in K$.


\begin{proposition}\label{FS} (Fiber shrinking)
Every $\mathcal{L}$-definable subset $A$ of $K^{n}$ with
accumulation point $a \in K^{n}$ has, after a permutation of the
coordinates, an $\mathcal{L}$-definable $x_{1}$-fiber shrinking at
$a$.
\end{proposition}

By means of elimination of valued field quantifiers
(Theorem~\ref{ball}), this proposition reduces easily to
Lemma~\ref{line-1} (similarly as it was in the algebraic case
treated in~\cite{Now-thm}). Now we can readily proceed with the

\vspace{1ex}

{\em Proof of the closedness theorem} (Theorem~\ref{clo-th}). We
must show that if $B$ is an $\mathcal{L}$-definable subset of $D
\times (K^{\circ})^{n}$ and a point $a$ lies in the closure of $A
:= \pi(B)$, then there is a point $b$ in the closure of $B$ such
that $\pi(b)=a$. As before (cf.~\cite[Section~~8]{Now-thm}), the
theorem reduces easily to the case $m=1$ and next, by means of
fiber shrinking (Proposition~\ref{FS}), to the case $n=1$. We may
obviously assume that $a = 0 \not \in A$.

\vspace{1ex}

By b-minimal cell decomposition, we can assume that the set $B$ is
a relative cell with center over $A$. It means that has a
presentation of the form
$$ \Lambda: B \ni (x,y) \to (x,\lambda(x,y)) \in  A \times RV(K)^{s}, $$
where $\lambda: B \to RV(K)^{s}$ is an $\mathcal{L}$-definable
function, such that for each $(x,\xi) \in \Lambda (B)$ the
pre-image $\lambda_{x}^{-1}(\xi) \subset K$ is either a point or
an open ball; here $\lambda_{x}(y) := \lambda(x,y)$. In the latter
case, there is a center, i.e.\ an $\mathcal{L}$-definable map
$\zeta: \Lambda(B) \to K$, and a (unique) map $\rho: \Lambda (B)
\to RV(K) \setminus \{ 0 \}$ such that
$$ \lambda_{x}^{-1}(\xi) = \{ y \in K: rv\, (y - \zeta (x,\xi)) =
   \rho (x,\xi) \} . $$
Again, since the sorts $\widetilde{K}$ and $\Gamma$ are orthogonal
in $RV(K) \simeq \widetilde{K} \times \Gamma_{K}$, we can assume,
after shrinking the sets $A$ and $B$, that
$$ \lambda(B) \subset \{ (1,\ldots,1) \} \times \Gamma^{s} \subset
   \widetilde{K}^{s} \times \Gamma_{K}^{s}; $$
let $\tilde{\lambda}(x,y)$ be the projection of $\lambda(x,y)$
onto $\Gamma^{s}$. By Lemma~\ref{line-2}, we can assume once
again, after shrinking the sets $A$ and $B$, that the set
$$ \{ (v(x),v(y),\tilde{\lambda}(x,y)): \; (x,y) \in B \} \subset
   \Gamma^{s+2} $$
is contained in an affine semi-line with integer coefficients.
Hence $\lambda(x,y) = \phi(v(x)$ is a function of one variable
$x$. We have two cases.

\vspace{1ex}

{\em Case I.} $\lambda_{x}^{-1}(\xi) \subset K^{\circ}$ is a
point. Since each $\lambda_{x}$ is a constant function, $B$ is the
graph of an $\mathcal{L}$-definable function. The conclusion of
the theorem follows thus from Proposition~\ref{limit-th1}.

\vspace{1ex}

{\em Case II.} $\lambda_{x}^{-1}(\xi) \subset K^{\circ}$ is a
ball. Again, application of Lemma~\ref{line-2} makes it possible,
after shrinking the sets $A$ and $B$, to arrange the center
$$ \zeta: \Lambda(B) \ni (x,k) \to \zeta(x, v(x)) = \zeta(x) \in K $$
and the function $\rho(x,k) = \rho(v(x))$ as functions of one
variable $x$. Likewise as it was above, we can assume that the set
$$ P := \{ (v(x), \rho(v(x))) : x \in A \} \subset \Gamma^{2} $$
is contained in an affine line $p v(x) +q \rho(v(x)) + v(c) =0$
with integer coefficients $p,q$, $q \neq 0$; furthermore, that $P$
contains the set
$$ Q := \{ (v(ct^{qN}), \rho(v(ct^{qN}))): \; t \in K^{\circ} \}
$$
for a positive integer $N$. Then we easily get
$$ \rho(v(ct^{qN})) = \rho(c) - pN v(t) = v(ct^{-pN}). $$
Hence the set $B$ contains the graph
$$ \{ (ct^{qN}, \zeta(ct^{qN}) + ct^{-pN}) : \; t \in K^{\circ} \}. $$
As before, the conclusion of the theorem follows thus from
Proposition~\ref{limit-th1}, and the proof is complete.

\section{Applications}

The framework of b-minimal structures provides cell decomposition
and a good concept of dimension (cf.~\cite{C-L}), which in
particular satisfies the axioms from the paper~\cite{Dries-dim}.
For separated analytic structures, the zero-dimensional sets are
precisely the finite sets, and also valid is the following
dimension inequality, which is of great geometric significance:
\begin{equation}\label{ineq}
\dim \, \partial E < \dim E;
\end{equation}
here $E$ is any $\mathcal{L}$-definable subset of $K^{n}$ and
$\partial E := \overline{E} \setminus E$ denotes the frontier of
$A$.

\vspace{1ex}

We first apply the closedness theorem to obtain a version of the
theorem on existence of the limit.

\begin{proposition}\label{limit-th2}
Let $f: E \to \mathbb{P}^{1}(K)$ be an $\mathcal{L}$-definable
function on a subset $E$ of $K$, and suppose that $0$ is an
accumulation point of $E$. Then there is a finite partition of $E$
into $\mathcal{L}$-definable sets $E_{1},\ldots,E_{r}$ and points
$w_{1}\ldots,w_{r} \in \mathbb{P}^{1}(K)$ such that
$$ \lim_{x \rightarrow 0}\, f|E_{j}\, (x) = w_{j} \ \ \ \text{for} \ \
   j=1,\ldots,r. $$
\end{proposition}

\begin{proof}

We may of course assume that $0 \not \in E$. Put
$$ F := \mathrm{graph}\, (f) = \{ (x, f(x): x \in E \} \subset K
   \times \mathbb{P}^{1}(K); $$
obviously, $F$ is of dimension $1$. It follows from the closedness
theorem that the frontier $\partial F \subset K \times
\mathbb{P}^{1}(K)$ is non-empty, and thus of dimension zero by
inequality~\ref{ineq}. Say
$$ \partial F \cap (\{ 0 \} \times \mathbb{P}^{1}(K)) = \{
   (0,w_{1}), \ldots, (0,w_{r}) \} $$
for some $w_{1},\ldots,w_{r} \in \mathbb{P}^{1}(K).$ Take pairwise
disjoint neighborhoods $U_{i}$ of the points $w_{i}$,
$i=1,\ldots,r$, and set
$$ F_{0} := F \cap \left(E \times \left( \mathbb{P}^{1}(K) \setminus \bigcup_{i}^{r}
   E_{i} \right) \right). $$
Let
$$ \pi: K \times \mathbb{P}^{1}(K) \longrightarrow K $$
be the canonical projection. Then
$$ E_{0} := \pi (F_{0}) = f^{-1}\left( \mathbb{P}^{1}(K) \setminus \bigcup_{i}^{r}
   E_{i} \right). $$
Clearly, the closure $\overline{F}_{0}$ of $F_{0}$ in $K \times
\mathbb{P}^{1}(K))$ and $\{ 0 \} \times \mathbb{P}^{1}(K))$ are
disjoint. Hence and by the closedness theorem, $0 \not \in
\overline{E_{0}}$, the closure of $E_{0}$ in $K$. The set $E_{0}$
is thus irrelevant with respect to the limes at $0 \in K$.
Therefore it remains to show that
$$ \lim_{x \rightarrow 0}\, f|E_{j}\, (x) = w_{j} \ \ \ \text{for} \ \
   j=1,\ldots,r. $$
Otherwise there is a neighborhood $V_{i} \subset U_{i}$ such that
$0$ would be an accumulation point of the set
$$ f^{-1}(U_{i} \setminus V_{i}) = \pi (F \cap (E
   \times (U_{i} \setminus V_{i}))). $$
Again, it follows from the closedness theorem that $\{ 0 \} \times
\mathbb{P}^{1}(K)$ and the closure of $F \cap (E \times (U_{i}
\setminus V_{i}))$ in $K \times \mathbb{P}^{1}(K))$ would not be
disjoint. This contradiction finishes the proof.
\end{proof}

\begin{remark}
Let us mention that Proposition~\ref{limit-th2} can be
strengthened as stated below (cf.\ the algebraic versions
\cite[Proposition~5.2]{Now-Sel} and \cite[Theorem~5.1]{Now-thm}):

\vspace{1ex}

\begin{em}
Moreover, perhaps after refining the finite partition of $E$,
there is a neighbourhood $U$ of $0$ such that each definable set
$$ \{ (v(x), v(f(x))): \; x \in (E_{j} \cap U) \setminus \{0 \} \}
   \subset \Gamma \times (\Gamma \cup \ \{
   \infty \}),  \ j=1,\ldots,r, $$
is contained in an affine line with rational slope
$$ l = \frac{p_{j}}{q} \cdot k + \beta_{j}, \ j=1,\ldots,r, $$
with $p_{j},q \in \mathbb{Z}$, $q>0$, $\beta_{j} \in \Gamma$, or
in\/ $\Gamma \times \{ \infty \}$.
\end{em}
\end{remark}

Now we turn to a second application, namely the following theorem
on piecewise continuity.

\begin{theorem}\label{piece}
Let $A \subset K^{n}$ and $f: A \to \mathbb{P}^{1}(K)$ be an
$\mathcal{L}$-definable function. Then $f$ is piecewise
continuous, i.e.\ there is a finite partition of $A$ into
$\mathcal{L}$-definable locally closed subsets
$A_{1},\ldots,A_{s}$ of $K^{n}$ such that the restriction of $f$
to each $A_{i}$ is continuous.
\end{theorem}

\begin{proof}
Consider an $\mathcal{L}$-definable function $f: A \to
\mathbb{P}^{1}(K)$ and its graph
$$ E := \{ (x,f(x)): x \in A \} \subset K^{n} \times \mathbb{P}^{1}(K). $$
We shall proceed with induction with respect to the dimension
$$ d = \dim A = \dim \, E $$
of the source and graph of $f$.

\vspace{1ex}

Observe first that every $\mathcal{L}$-definable subset $E$ of
$K^{n}$ is a finite disjoint union of locally closed
$\mathcal{L}$-definable subsets of $K^{n}$. This can be easily
proven by induction on the dimension of $E$ by means of
inequality~\ref{ineq}. Therefore we can assume that the graph $E$
is a locally closed subset of $K^{n} \times \mathbb{P}^{1}(K)$ of
dimension $d$ and that the conclusion of the theorem holds for
functions with source and graph of dimension $< d$.

\vspace{1ex}

Let $F$ be the closure of $E$ in $K^{n} \times \mathbb{P}^{1}(K)$
and $\partial E := F \setminus E$ be the frontier of $E$. Since
$E$ is locally closed, the frontier $\partial E$ is a closed
subset of $K^{n} \times \mathbb{P}^{1}(K)$ as well. Let
$$ \pi: K^{n} \times \mathbb{P}^{1}(K) \longrightarrow K^{n} $$
be the canonical projection. Then, by virtue of the closedness
theorem, the images $\pi(F)$ and $\pi(\partial E)$ are closed
subsets of $K^{n}$. Further,
$$ \dim \, F = \dim \, \pi(F) = d $$
and
$$ \dim \, \pi(\partial E) \leq \dim \, \partial E < d; $$
the last inequality holds by inequality~\ref{ineq}. Putting
$$ B := \pi(F) \setminus \pi(\partial E) \subset \pi(E) = A, $$
we thus get
$$ \dim \, B = d \ \ \text{and} \ \ \dim \, (A \setminus B) < d.
$$
Clearly, the set
$$ E_{0} := E \cap (B \times \mathbb{P}^{1}(K)) = F \cap (B \times
   \mathbb{P}^{1}(K)) $$
is a closed subset of $B \times \mathbb{P}^{1}(K)$ and is the
graph of the restriction
$$ f_{0}: B \longrightarrow \mathbb{P}^{1}(K) $$
of $f$ to $B$. Again, it follows immediately from the closedness
theorem that the restriction
$$ \pi_{0} : E_{0} \longrightarrow B $$
of the projection $\pi$ to $E_{0}$ is a definably closed map.
Therefore $f_{0}$ is a continuous function. But, by the induction
hypothesis, the restriction of $f$ to $A \setminus B$ satisfies
the conclusion of the theorem, whence so does the function $f$.
This completes the proof.
\end{proof}

We immediately obtain

\begin{corollary}
The conclusion of the above theorem holds for any
$\mathcal{L}$-definable function $f: A \to K$.
\end{corollary}


Let us conclude with the following comment. We are currently
preparing subsequent articles, which will provide several
applications of the closedness theorem, possibly over
non-algebraically closed ground fields, including i.al.\ the
analytic, non-Archimedean versions of the \L{}ojasiewicz
inequalities and of curve selection. The algebraic versions of
these results were established in our
papers~\cite{Now-Sel,Now-thm}.

\vspace{3ex}

\begin{small}
Institute of Mathematics

Faculty of Mathematics and Computer Science

Jagiellonian University


ul.~Profesora S.\ \L{}ojasiewicza 6

30-348 Krak\'{o}w, Poland

{\em E-mail address: nowak@im.uj.edu.pl}
\end{small}

\end{document}